\documentclass[11pt,reqno]{amsart}
\usepackage{amssymb,mathtools,calc,verbatim,enumitem,tikz,url,hyperref,mathrsfs,cite,fullpage}
\usepackage{bbm}
\usepackage{stmaryrd}
\usepackage{textcomp}
\usepackage{setspace}
\usepackage{amssymb}
\usepackage{amsthm}
\usepackage{amsmath}
\usepackage{graphicx}
\usepackage{marvosym}
\usepackage{empheq}
\usepackage{latexsym}
\usepackage{fontenc}
\usepackage{color}
\usepackage{hyperref}
\usepackage{cleveref}
\usepackage{dsfont}

\usepackage{todonotes}

\newenvironment{poc}{\begin{proof}[Proof of claim]}{\end{proof}}

\newtheorem{theorem}{Theorem}[section]
\newtheorem{lemma}[theorem]{Lemma}

\newtheorem{conjecture}[theorem]{Conjecture}

\newtheorem{proposition}[theorem]{Proposition}
\newtheorem{claim}[theorem]{Claim}
\newtheorem*{claim*}{Claim}

\theoremstyle{definition}

\newtheorem*{qu*}{Question}
\theoremstyle{remark}
\newtheorem{remark}[theorem]{Remark}

\newcommand\cM{\mathcal{M}}

\renewcommand\leq{\leqslant}
\renewcommand\geq{\geqslant}
\renewcommand\le{\leqslant}
\renewcommand\ge{\geqslant}

	\def\<{\langle }
	\def\>{\rangle }

\begin{document}

\title{Colour diversity in spanning structures under Dirac-type conditions}
\author{
Xinbu Cheng 
\and 
Xinqi~Huang
\and  
Hong Liu 
\and 
Bin Wang 
\and 
Zhifei Yan}

\address{IMPA, Estrada Dona Castorina 110, Jardim Bot\^anico,
Rio de Janeiro, 22460-320, Brazil}\email{xinbu.cheng@impa.br}

\address{University of Science and Technology of China,
Hefei, 230026, China and ECOPRO, Institute for Basic Science, 55 Expo-ro, Yuseong-gu, Daejeon, 34126, Korea}\email{huangxq@mail.ustc.edu.cn}

\address{ECOPRO, Institute for Basic Science, 55 Expo-ro, Yuseong-gu, Daejeon, 34126, Korea}\email{\{hongliu,zhifeiyan\}@ibs.re.kr}

\address{School of Mathematics and Statistics, Beijing Institute of Technology, China}
\email{bin.wang@bit.edu.cn}

\thanks{Xinqi Huang, Hong Liu and Zhifei Yan were supported by the Institute for Basic Science (IBS-R029-C4)}
	
\begin{abstract}
Finding spanning structures with many distinct colours in properly edge-coloured graphs is a central theme in extremal combinatorics. A classical result of Andersen shows that every proper edge-colouring of the complete graph $K_n$ contains a Hamilton cycle with $n - O(n^{1/2})$ distinct colours. In the bipartite setting, the analogous question for perfect matchings is closely related to permutations in Latin squares. In this paper, we investigate how a Dirac-type minimum degree condition forces colour diversity in spanning structures. For every constant $1/2 < c \le 1$, we prove the following.
\begin{itemize}
    \item Every properly edge-coloured graph $G$ on $n$ vertices with $\delta(G)\ge cn$ contains a Hamilton cycle with at least $cn - O(n^{1/2})$ distinct colours.

    \item Every subset of an $n\times n$ Latin square with at least $cn$ cells in each row and each column contains a permutation with at least $cn - O(n^{2/3})$ distinct symbols.
\end{itemize}
Both bounds are best possible up to the error term.
\end{abstract}

	\maketitle

\section{Introduction}
Given a properly edge-coloured graph, a natural and long-studied problem is to find a spanning structure whose edges all receive different colours, commonly referred to as a \emph{rainbow} structure. More generally, when a fully rainbow structure does not exist, one seeks spanning structures that maximise the number of distinct colours they contain. In this paper we investigate how minimum degree conditions force colour diversity in spanning structures. This line of research has been particularly active for Hamilton cycles in graphs and perfect matchings in bipartite graphs, two central spanning structures in extremal combinatorics.

\subsection{Hamilton cycles in graphs}

In 1980, Hahn~\cite{Hahn} conjectured that every properly edge-coloured
$K_n$ with $n\ge5$ contains a rainbow Hamilton path.
This conjecture was disproved by Maamoun and Meyniel~\cite{MM}, who
constructed infinitely many counterexamples.
This shifted attention to the weaker but more robust question of how many
distinct colours can be guaranteed on a Hamilton cycle. A foundational result in this direction is due to Andersen~\cite{A}, who proved that properly edge-coloured complete graphs contain
almost rainbow Hamilton cycles.

\begin{theorem}[Andersen]
For all sufficiently large $n$, every properly edge-coloured $K_n$
contains a Hamilton cycle with at least $n-\sqrt{2n}$ distinct colours.
\end{theorem}

Balogh and Molla~\cite{BM} later improved the error term, proving that every properly edge-coloured $K_n$ contains a Hamilton cycle with at least $n - O(\log^2 n)$ distinct colours.

More recently, rainbow Hamiltonicity has been investigated in general
graphs.
Coulson and Perarnau~\cite{CP} established a rainbow Dirac theorem under
a global boundedness condition on colour classes.
Peng and the fifth author~\cite{PY} subsequently studied proper
edge-colourings and showed that if each colour appears on at most $n/8$
edges, then every Dirac graph contains a Hamilton cycle with $n-o(n)$
distinct colours.
Without any boundedness assumption, they obtained a universal bound of
$n/4-o(n)$ colours.

These results highlight a difficulty: unlike the complete
graph, a properly edge-coloured Dirac graph with $\delta(G)\ge cn$ may use only $cn+O(1)$ colours. In particular, even the total number of colours available may be far below $n$, so a rainbow Hamilton cycle cannot in general be expected.
As a result, techniques that rely on an abundance of colours in $K_n$ do
not directly extend to the min-degree setting. Motivated by this difficulty, the following conjecture was proposed in~\cite{PY}.
\begin{conjecture}[\cite{PY}]\label{conj:PY}
   Every proper edge-colouring of an $n$-vertex graph with $\delta(G)\ge n/2$ contains a Hamilton cycle with at least $n/2-o(n)$ distinct colours.   
\end{conjecture}

We verify this conjecture for all graphs with min-degree $cn$ with $c>1/2$.

\begin{theorem}\label{thm:main}
Let $1/2<c\le1$ and let $n$ be sufficiently large.
If $G$ is an $n$-vertex graph with $\delta(G)\ge cn$, then every properly
edge-coloured $G$ contains a Hamilton cycle with at least
$cn-O(\sqrt n)$ distinct colours.
\end{theorem}

This bound is asymptotically optimal up to the $O(\sqrt n)$ term. Our result may also be viewed from the perspective of \emph{resilience}.
In general, resilience-type problems ask to what extent a global graph property
persists under local perturbations, and they have been studied extensively in
extremal and random graph theory.
From this viewpoint, Andersen’s theorem treats the extremal case of the complete
graph, while Theorem~\ref{thm:main} shows that the ``nearly all available colours appear on a Hamilton cycle'' phenomenon is stable under arbitrary edge deletions that preserve the Dirac condition.

The core technical contribution is a new structural result for properly edge-coloured Dirac graphs: the existence of a
spanning linear forest with few components that already uses a
near-maximal number of colours.

\begin{theorem}\label{thm:pathforest}
Let $1/2<c\le1$ and let $n$ be sufficiently large.
Every properly edge-coloured $n$-vertex graph $G$ with $\delta(G)\ge cn$
contains a spanning linear forest with at most $\sqrt n$ components and at
least $cn-4\sqrt n$ distinct colours.
\end{theorem}

Our construction builds on the classical backtracking algorithm of
Andersen~\cite{A} and Alon, Pokrovskiy, and Sudakov~\cite{APS}, but new ideas
are needed in the Dirac setting to compensate for the potential scarcity of
colours.
The key new ingredient is a rigidity phenomenon for maximal colourful
linear forests. Roughly speaking, once such a forest cannot be extended while preserving
its colour set, many of its edges acquire a form of global uniqueness:
their colours cannot be replicated elsewhere without destroying maximality.
This structural constraint allows us to iteratively expose many new
vertices and colours, leading to a linear growth of the set of reachable
colours.
This rigidity mechanism appears to be specific to the Dirac setting and may be of independent interest.

It is tempting to combine absorption with recent results on long rainbow paths.
For example, Buci\'c, Frederickson, M\"uyesser, Pokrovskiy, and Yepremyan~\cite{BFMPY}
showed that every properly edge-coloured graph $G$ with $\delta(G)\ge cn$ contains a rainbow
path of length $cn-o(n)$. However, long colourful paths (or linear forests) need not be compatible with Hamilton
cycles: even in Dirac graphs there exist linear forests with $cn$ edges such that no
Hamilton cycle contains more than a $(2c-1)$-fraction of their edges, as the following proposition shows. This obstruction explains why one must carefully control the global structure of the forest rather than merely maximize its size.

\begin{proposition}\label{Prop:construction}
Let $1/2 < c < 1$ and let $n \in \mathbb{N}$ be sufficiently large. There exists a graph $G$ on $n$ vertices with $\delta(G) \ge cn$ and a linear forest $F$ in $G$ with $cn$ edges such that every Hamilton cycle in $G$ contains at most $(2c - 1)n$ edges of $F$.
\end{proposition}

\subsection{Perfect matchings in bipartite graphs}

Another well-studied direction concerns perfect matchings in properly edge-coloured $K_{n,n}$.
This problem is equivalent to a classical question concerning permutations and transversals in Latin squares.
An $n \times n$ Latin square is an $n \times n$ array filled with $n$ symbols such that each symbol appears exactly once in every row and every column.
A permutation in a Latin square is a selection of $n$ cells such that no two cells share the same row or column.
If, in addition, the symbols appearing in these cells are all distinct, then such a permutation is called a transversal.
The existence of transversals in Latin squares was highlighted by a conjecture of Ryser~\cite{Ryser}, which asserts that every Latin square of odd order contains a transversal.
Subsequently, Brualdi and Stein~\cite{BR,Stein} formulated a stronger conjecture, proposing that every Latin square of order $n$ admits a partial transversal of size at least $n-1$.


A sequence of works has established the existence of large partial transversals in Latin squares.
Brouwer, A. de Vries, and Wieringa~\cite{BVW} proved that every Latin square of order $n$ contains a partial transversal of size $n - O(\sqrt{n})$.
Hatami and Shor~\cite{HS} improved this to $n - O(\log^2 n)$.
Keevash, Pokrovskiy, Sudakov, and Yepremyan~\cite{KPSY} reduced the error term to $n - O(\log n / \log\log n)$.
Very recently, Montgomery~\cite{Montgomery} proved the existence of transversals of size $n - 1$ for every sufficiently large $n$.

Our next result finds almost maximally colourful perfect matchings in bipartite graphs, which may be viewed as a natural min-degree analogue of the Ryser--Brualdi--Stein problem.

\begin{theorem}\label{thm:bipartite}
Let $1/2 < c \le 1$ and let $n\in\mathbb{N}$ be sufficiently large. Every properly edge-coloured balanced bipartite graph $G$ on $2n$ vertices with $\delta(G)\geq cn$ contains a perfect matching with at least
$cn - O(n^{2/3})$ distinct colours.
\end{theorem}

Equivalently,~\Cref{thm:bipartite} reads as follows. Every subset of an $n\times n$ Latin square with at least $cn$ cells in each row and column contains a permutation with at least $cn - O(n^{2/3})$ distinct symbols.

Our proof of Theorem~\ref{thm:bipartite} is motivated by the work of Gy\'{a}rf\'{a}s and S\'{a}rk\"{o}zy~\cite{GS}, who proved that every properly edge-coloured graph $G$ with $\delta(G) \le n/2$ contains a rainbow matching of size $\delta(G)-o(n)$.
However, a matching of size $\delta(G)-o(n)$ is not necessarily extendable to a perfect matching, and there does not always exist a perfect matching that contains all but $o(n)$ edges of such a matching, as shown below.

\begin{proposition}\label{Prop:construction2}
Let $1/2 < c < 1$ and let $n \in \mathbb{N}$ be sufficiently large. There exists a balanced bipartite graph $G$ on $2n$ vertices with $\delta(G) \ge cn$ and a matching $M\subset G$ with $cn$ edges such that every perfect matching in $G$ contains at most $(2c - 1)n$ edges of $M$.
\end{proposition}

\noindent\textbf{Organisation of the paper.} The remainder of this paper is organised as follows.
In Section~\ref{sec:Hamiltoncycle}, we establish the path-forest Theorem~\ref{thm:pathforest} and then combine it with an absorption argument to prove Theorem~\ref{thm:main}.
In Section~\ref{sec:perfectmatching}, we prove Theorem~\ref{thm:bipartite}.
In Section~\ref{sec:badconstruction}, we present the constructions for Propositions~\ref{Prop:construction} and~\ref{Prop:construction2}.
Concluding remarks are in Section~\ref{sec:conclusion}.

\medskip

\section{Colourful Hamilton cycles in Dirac graphs}\label{sec:Hamiltoncycle}

In this section, we prove Theorems~\ref{thm:main} and \ref{thm:pathforest}. We first prove the existence of a spanning linear forest in a properly edge-coloured Dirac graph that uses almost as many distinct colours as possible (see Proposition~\ref{prop:pathforest}), which directly implies
Theorem~\ref{thm:pathforest}. Then combining with absorption method for Hamilton cycles, we prove Theorem~\ref{thm:main}.

\subsection{Colourful path-forests}
In this subsection we prove the key structural proposition for path-forest, as follows.

\begin{proposition}\label{prop:pathforest}
For $1/2<c\le 1$, there exists $n_0\in\mathbb N$ such that for all $n\ge n_0$ the following holds. Let $\delta=\delta(n)$ and $t=t(n)$ satisfy $2c<t\delta<\delta^2 n/2$. If $G$ is a graph on $n$ vertices with $\delta(G)\ge cn$ and $\chi$ is a proper edge-colouring of $G$, then $G$ contains a spanning linear forest $F$ consisting of $t$ vertex-disjoint paths and containing at least $(c-\delta)n$ distinct colours.
\end{proposition}

Proposition~\ref{prop:pathforest} directly implies Theorem~\ref{thm:pathforest}, by choosing $t=n^{1/2}$ and $\delta=4n^{-1/2}$.

The proof proceeds by contradiction. Consider a spanning linear forest with $t$ components that maximises the number of distinct colours. Suppose that this forest uses less than $(c-\delta)n$ distinct colours, we study how unused colours can be propagated into the forest via its path endpoints. This leads to an iterative expansion of a family of colour sets $C_1 \subseteq \cdots \subseteq C_t$, where at each step new colours are ``reached'' through coloured neighbourhoods of an endpoint. The key observation is that maximality forces a strong rigidity: each such reachable vertex contributes a globally unique colour among the forest edges, since any collision would allow a backtracking modification that strictly increases the number of colours. Using the minimum degree condition and properness of the colouring, this forces the size of reachable colour sets to grow linearly, ultimately contradicting the maximality of the forest.

Let $G$ be a graph on $n$ vertices with $\delta(G)\ge cn$ where $1/2<c\le 1$, equipped with a proper
edge-colouring $\chi$.
Fix an integer $t$, and let
\begin{equation*}
\mathcal F=\big\{F\subset G:\; F \text{ is a spanning linear forest with exactly $t$ paths}\big\}.
\end{equation*}
By Dirac's theorem, $G$ contains a Hamilton cycle, and hence $\mathcal F\neq\varnothing$ for all $t\le n/2$. Choose a spanning linear forest $F_0\in\mathcal F$ that maximises the number of distinct colours appearing on its edges. Suppose that $F_0$ uses fewer than $(c-\delta)n$ distinct colours.

Write $F_0=\bigcup_{i=1}^t P_i$, where 
\[
P_i = v_1^{(i)}v_2^{(i)}\cdots v_{m_i}^{(i)}
\]
are vertex-disjoint paths. For a vertex $v_j^{(i)}$ with $1\le i\le t$ and $2\le j\le m_i$, let
\[
f\big(v_j^{(i)}\big) := \chi\big(v_{j-1}^{(i)}v_j^{(i)}\big)
\]
be the colour of the edge connecting $v_j^{(i)}$ to its predecessor on the path.
We refer to $v_{j-1}^{(i)}$ as the \emph{prefix} of $v_j^{(i)}$ and write
$P(v_j^{(i)})=v_{j-1}^{(i)}$.

Let $C$ denote the set of all colours used by $\chi$, and let $C_0$ be the set of colours that do not appear in $F_0$.
We now define an increasing sequence of colour sets
\[
C_0\subseteq C_1\subseteq \cdots \subseteq C_t
\]
as follows.
For each $i\in[t]$, let
\begin{equation}\label{eq:colourset}
C_i :=
\Big\{ f(x)\;:\;
x\in N_{C_{i-1}}\!\big(v_1^{(i)}\big)
      \setminus \{v_1^{(1)},\ldots,v_1^{(t)}\}
\Big\}
\;\cup\; C_{i-1},
\end{equation}
where $N_{C_{i-1}}(v)=\{u:uv\in E(G){\rm \ and}\ \chi(uv)\in C_{i-1}\}$.

By construction, $C_i\setminus C_0$ consists of colours appearing on edges of $F_0$.
In particular,
\[
C_t\setminus C_0 \subseteq C\setminus C_0,
\]
so it suffices to show that $|C_t\setminus C_0|>(c-\delta)n$ in order to contradict the
maximality of $F_0$.


The key observation is the following lemma. Whenever a vertex $x$ can be reached from a path endpoint via colours already
available at an earlier stage, the colour of the forest-edge incident to $x$ must
be \emph{globally unique} in $F_0$. This rigidity is the core mechanism that drives the growth of the colour sets
$C_i$.

\begin{lemma}\label{lem:keyobs}
Let $F_0\in\mathcal F$ and $C_i$, $i=0,1,\ldots,t$ be defined as above. 
Suppose that the number of colours in $F_0$ is less than $(c-\delta)n$.
Then for every $1\le i\le t$ and every
\[
x\in N_{C_{i-1}}\!\big(v_1^{(i)}\big)\setminus \{v_1^{(1)},\ldots,v_1^{(t)}\} \quad\text{and}\quad
y\in V(F_0)\setminus\{x,v_1^{(1)},\ldots,v_1^{(t)}\},
\]
we have $f(x)\neq f(y)$.
\end{lemma}

\begin{proof}
The proof is by induction on $i$.
The idea is that any colour collision of the type described in the
statement would allow us to locally modify the forest to introduce a new
colour from $C_0$ while preserving all existing colours, contradicting the
maximality of $F_0$.

Since the colours of edges incident to $v_1^{(i)}$ are all distinct by the properness of $\chi$, we have 
\begin{align*}
    |N_{C_0}(v_1^{(i)})\setminus \{v_1^{(1)},\ldots,v_1^{(t)}\}|\ge cn - (c-\delta)n -t \ge \frac{\delta }{2}n,
\end{align*}
which implies that $N_{C_{0}}\!\big(v_1^{(i)})\neq\emptyset$.
For the base case $i=1$, we prove a slightly stronger statement that for every $i\in[t]$ and every $x\in N_{C_{0}}\!\big(v_1^{(i)}\big)\setminus \{v_1^{(1)},\ldots,v_1^{(t)}\}$ and 
$y\in V(F_0)\setminus\{x,v_1^{(1)},\ldots,v_1^{(t)}\}$, we have $f(x)\neq f(y)$. Indeed, if there exists a vertex $x\in N_{C_{0}}\!\big(v_1^{(i)}\big)\setminus \{v_1^{(1)},\ldots,v_1^{(t)}\}$ such that $f(x)=f(y)$ for some $y\in V(F_0)\setminus\{x,v_1^{(1)},\ldots,v_1^{(t)}\}$, denote this vertex $x$ by $v_j^{(p)}$ for some $p\in[t]$ and $j\in[2,m_p]$.
Construct $F_1$ from $F_0$ by deleting the $F_0$-edge $xv_{j-1}^{(p)}$ and adding the edge
$v_1^{(i)}x$.
Note that $F_1$ is still a spanning linear forest with exactly $t$ components. The added edge $v_1^{(i)}x$ has colour in $C_0$, hence introduces a new colour.
Removing $xv_{j-1}^{(p)}$ does not decrease the number of colours, since the same
colour $f(x)$ still appears on the edge incident to $y$.
Thus $F_1$ uses more colours than $F_0$, contradicting the maximality of $F_0$.

Assume that the conclusion of the lemma holds for all indices $\ell<i$.
Suppose for contradiction that there exist vertices
\[
x\in N_{C_{i-1}}\!\big(v_1^{(i)}\big)\setminus\{v_1^{(1)},\ldots,v_1^{(t)}\}
\quad\text{and}\quad
y\in V(F_0)\setminus\{x,v_1^{(1)},\ldots,v_1^{(t)}\}
\]
with $f(x)=f(y)$. Note that by the stronger statement we proved in the base case, we have $\chi(v_1^{(i)}x)\in C_{i-1}\setminus C_0$. 
By the definition of $C_{i-1}$ (see \eqref{eq:colourset}), we have $\chi(v_1^{(i)}x)=f(x_1)$ for some vertex
\[
x_1\in N_{C_{s_1-1}}\!\big(v_1^{(s_1)}\big)\setminus \{v_1^{(1)},\ldots,v_1 ^{(t)}\},
\]
where $s_1<i$.

Similarly, we must have $\chi(v_1^{(s_1)}x_1)\in C_{s_1-1}\setminus C_0$, and hence we may find a vertex 
\[
x_2\in N_{C_{s_2-1}}\!\big(v_1^{(s_2)}\big)\setminus \{v_1^{(1)},\ldots,v_1 ^{(t)}\}
\]
for some $s_2<s_1$ such that $\chi\big(v_1^{(s_1)}x_1\big)=f(x_2).$

Repeating this argument, we obtain a strictly decreasing sequence of indices
\[
i=s_0>s_1>\cdots>s_m=1
\]
and vertices
\[
x=x_0,x_1,\ldots,x_m \in V(F_0)\setminus \{v_1^{(1)},\ldots,v_1 ^{(t)}\}
\]
such that $x_m\in N_{C_0}\!\big(v_1^{(1)}\big)$.

By the induction hypothesis, the vertices $x_0,\ldots,x_m$ are all distinct, since
any repetition would violate the uniqueness of colours at an earlier stage.
This implies that the colours $f(x_0),\ldots,f(x_m)$ are all distinct.

We now modify $F_0$ as follows.
For each $0\le\ell\le m$, delete the unique forest edge incident to $x_\ell$ whose
colour is $f(x_\ell)$, and add the edge $x_\ell v_1^{(s_\ell)}$.
The resulting graph $F_1$ is again a spanning linear forest with $t$ components.

For each $\ell<m$, the added edge $x_\ell v_1^{(s_\ell)}$ has the same colour as the
deleted edge incident to $x_{\ell+1}$, so no colour is lost.
At the final step, the edge $x_m v_1^{(1)}$ introduces a colour from $C_0$.
Thus $F_1$ uses strictly more colours than $F_0$, contradicting the maximality of $F_0$.

This contradiction completes the induction and proves the lemma.
\end{proof}

We are now ready to prove Proposition~\ref{prop:pathforest}.

\begin{proof}[Proof of Proposition~\ref{prop:pathforest}]
Let $F_0=\bigcup_{i=1}^t P_i\in\mathcal F$ be a spanning linear forest with exactly $t$
components that maximises the number of distinct colours, and suppose for
contradiction that $F_0$ uses fewer than $(c-\delta)n$ colours.
Let $C_0\subset C_1\subset\cdots\subset C_t$ be the colour sets defined in
\eqref{eq:colourset}, and let $C(F_0)=\chi(E(F_0))$ denote the set of colours appearing
in $F_0$.
Since $C_t\setminus C_0\subseteq C(F_0)$, it suffices to show that
\[
|C_t\setminus C_0|>(c-\delta)n,
\]
which would contradict the maximality of $F_0$.

We first establish a lower bound on coloured neighbourhoods. Fix $1\le i\le t$ and a vertex $v\in V(G)$.
Since $\chi$ is a proper edge-colouring, all edges incident to $v$ have distinct
colours.
An edge $vw$ is not counted in $N_{C_i}(v)$ only if its colour lies in
$(C(F_0)\cup C_0)\setminus C_i$.
Therefore,
\[
|N_{C_i}(v)|
\ge |N(v)| - |(C(F_0)\cup C_0)\setminus C_i|.
\]
Because $F_0$ is spanning, we have $V(F_0)=V(G)$.
Moreover, $C_0\subset C_i$ and $C_0\cap C(F_0)=\varnothing$, so
\[
|(C(F_0)\cup C_0)\setminus C_i|
=|C(F_0)|-|C_i\setminus C_0|.
\]
Using $\delta(G)\ge cn$ and $|C(F_0)|\le(c-\delta)n$, we obtain
\begin{equation}\label{eq:neighbour-lower}
|N_{C_i}(v)|\ge cn-|C(F_0)|+|C_i\setminus C_0|
\ge \delta n+|C_i\setminus C_0|.
\end{equation}

We next show the growth of the colour sets $C_i\setminus C_0$.
Fix $1\le i\le t$.
By the definition of $C_i$, we have
\[
C_i\setminus C_0
\supseteq
\Big\{f(x):x\in N_{C_{i-1}}(v_1^{(i)})\setminus\{v_1^{(1)},\ldots,v_1^{(t)}\}\Big\}.
\]
By Lemma~\ref{lem:keyobs}, all colours $f(x)$ appearing in the above set are distinct.
Hence
\begin{equation}\label{eq:colour-growth}
|C_i\setminus C_0|
\ge |N_{C_{i-1}}(v_1^{(i)})|-t.
\end{equation}

Applying \eqref{eq:neighbour-lower} with $v=v_1^{(i)}$ and $i-1$ in place of $i$, we get
\[
|N_{C_{i-1}}(v_1^{(i)})|
\ge |C_{i-1}\setminus C_0|+\delta n.
\]
Combining this with \eqref{eq:colour-growth} yields the recurrence
\begin{equation}\label{eq:recurrence}
|C_i\setminus C_0|
\ge |C_{i-1}\setminus C_0|+(\delta n-t)
\qquad\text{for all }1\le i\le t.
\end{equation}

Summing \eqref{eq:recurrence} over $i=1,\ldots,t$, we obtain
\[
|C_t\setminus C_0|
=\sum_{i=1}^t\big(|C_i\setminus C_0|-|C_{i-1}\setminus C_0|\big)
\ge t(\delta n-t).
\]
By assumption, $t<\delta n/2$, and hence $\delta n-t>\delta n/2$.
Together with $t\delta>2c$, this implies
\[
|C_t\setminus C_0|
> \frac{t\delta n}{2}
> (c-\delta)n,
\]
contradicting the assumption that $F_0$ uses fewer than $(c-\delta)n$ colours.

This contradiction completes the proof.
\end{proof}

\subsection{Proof of Theorem~\ref{thm:main}}

In this subsection we complete the proof of Theorem~\ref{thm:main} by turning the
colourful linear forest obtained in Proposition~\ref{prop:pathforest} into a
Hamilton cycle.
Our argument follows the standard \emph{absorption method} introduced
in~\cite{RRS}. The strategy is as follows.
First, we construct a short \emph{absorbing path} $A$ with the property that any
small set of vertices can later be incorporated into $A$ while preserving its
endpoints.
Second, we choose a small \emph{reservoir} set $R$ that allows us to connect many
vertex-disjoint paths.
After removing $A\cup R$, we apply Proposition~\ref{prop:pathforest} to the
remaining graph to obtain a spanning linear forest with many colours.
Finally, we use vertices from $R$ to merge all paths into a single cycle and then
absorb the remaining vertices via $A$.
Throughout this process, the number of distinct colours on the colourful forest is preserved.

We begin with the absorbing lemma, which is Lemma~3.1 from~\cite{PY}.

\begin{lemma}[Absorber for Hamilton cycles]\label{lem:absorber}
For every $0<\varepsilon\leq 1/2$ and $0<d<2^{-10}\varepsilon$, there exist constants
$C=C(d,\varepsilon)$ and $n_0=n_0(d,\varepsilon)$ such that the following holds
for all $n\ge n_0$ and all integers $m$ with $(\log n)^2\le m\le n/C$.
If $G$ is an $n$-vertex graph with $\delta(G)\ge (1/2+\varepsilon)n$, then there
exists a path $A\subset G$ with
\[
|V(A)|\le Cm
\]
such that for every set $U\subset V(G)\setminus V(A)$ with $|U|\le dm$, there
exists a path $A_U\subset G$ with the same endpoints as $A$ and
$V(A_U)=V(A)\cup U$.
\end{lemma}

Next we state the reservoir lemma, which will be used to connect the components
of a linear forest, which is Lemma~3.3 from~\cite{PY}.

\begin{lemma}[Reservoir for Hamilton cycles]\label{lem:reservoir}
For every $0<d,\varepsilon\leq 1/2$ and $C>1$, there exists $n_0=n_0(d,\varepsilon,C)$
such that the following holds for all $n\ge n_0$ and all integers
$(\log n)^2\le m\le \varepsilon n/(2d)$.
If $G$ is an $n$-vertex graph with $\delta(G)\ge (1/2+\varepsilon)n$ and
$W\subset V(G)$ satisfies $|W|\le Cm$, then there exists a set
$R\subset V(G)\setminus W$ with $|R|=dm$ such that
\[
|N(x)\cap N(y)\cap R|\ge \varepsilon|R|
\]
for every pair of vertices $x,y\in V(G)$.
\end{lemma}

We are now ready to prove Theorem~\ref{thm:main}.

\begin{proof}[Proof of Theorem~\ref{thm:main}]
Let $G$ be an $n$-vertex graph with $\delta(G)\ge cn$, equipped with a proper
edge-colouring $\chi$.
Set
\[
\varepsilon=(c-1/2)/2,\qquad
d=2^{-11}\varepsilon,\qquad
m=2^{12}\varepsilon^{-2}n^{1/2}.
\]

Applying Lemma~\ref{lem:absorber} to $G$, we obtain an absorbing path
$A\subset G$ with $|V(A)|\le Cm=Dn^{1/2}$
for some constant $D=D(\varepsilon)$, such that any vertex set
$U\subset V(G)\setminus V(A)$ with $|U|\le dm$ can be absorbed into $A$.

Next, we apply Lemma~\ref{lem:reservoir} with $W=V(A)$ to obtain a
reservoir set
\[
R\subset V(G)\setminus V(A)
\quad\text{with}\quad
|R|=dm=2\varepsilon^{-1}n^{1/2}
\]
such that
\begin{equation}\label{eq:common}
|N(x)\cap N(y)\cap R|\ge \varepsilon|R|\ge 2n^{1/2}
\end{equation}
for all $x,y\in V(G)$.

Let $G' = G\big[V(G)\setminus (V(A)\cup R)\big]$,
and let $\chi'$ be the restriction of $\chi$ to $E(G')$.
Note that $\chi'$ is a proper edge-colouring of $G'$ and 
$$\delta(G')\ge cn-|V(A)|-|R|
\ge cn-(D+2\varepsilon^{-1})n^{1/2}.$$ Applying Theorem~\ref{thm:pathforest} to $G'$, we obtain a spanning linear forest
$F\subset G'$ with at most $n^{1/2}$ components and at least $cn-(2D+4\varepsilon^{-1}+4)n^{1/2}$
distinct colours.

We now merge the components of $F$ and the path $A$ into a single cycle using
vertices from $R$.
Write $F=P_1\cup\cdots\cup P_T$, where $T\le n^{1/2}$ and each $P_i$ has endpoints
$\{x_i,y_i\}$.
Let $x_0,y_0$ be the endpoints of $A$, and set $x_{T+1}=x_0$.
By~\eqref{eq:common}, for each $0\le i\le T$ the vertices $y_i$ and $x_{i+1}$
have at least $2n^{1/2}$ common neighbours in $R$.
Since $T+1<2n^{1/2}$, we can greedily choose distinct vertices
\[
z_i\in N(y_i)\cap N(x_{i+1})\cap R
\quad (0\le i\le T).
\]

Adding the edges $\{y_iz_i,z_ix_{i+1}:0\le i\le T\}$ produces a cycle $H$ with
\[
|V(H)|\ge |V(F)|+|V(A)|\ge n-|R|.
\]
Importantly, this step only adds edges without deleting any edges from $F$, so $H$ retains at least as many distinct colours as $F$.

Finally, let $U=V(G)\setminus V(H)$ and we have $|U|\le |R|\le 2\varepsilon^{-1}n^{1/2}$.
By Lemma \ref{lem:absorber}, we obtain a path $A_U$ having the same endpoints with $A$ and
$V(A_U)=V(A)\cup U$.
Replacing $A$ by $A_U$ in $H$ yields a Hamilton cycle in $G$ with at least
$cn-O(n^{1/2})$ distinct colours.
\end{proof}

\medskip

\section{Colourful perfect matchings in Dirac bipartite graphs}\label{sec:perfectmatching}

In this section we prove Theorem~\ref{thm:bipartite}. Let $G$ be a balanced bipartite graph on $2n$ vertices with $\delta(G)\ge cn$ and equipped with a proper edge-colouring. We first construct a matching in $G$ with $n-o(n)$ edges that uses $cn-o(n)$ distinct colours (see Proposition~\ref{prop:matching}), and then extend this matching to a perfect one retaining most of its colours.

\subsection{Colourful near-perfect matchings}

In this subsection, our goal is to construct a near-perfect matching with a near-optimal number of colours, as follows.

\begin{proposition}\label{prop:matching}
For $1/2<c\le 1$, there exists $n_0\in\mathbb N$ such that for all $n\ge n_0$ the following holds. Let $t=t(n)\geq n^{2/3}$. If $G$ is a balanced bipartite graph on $2n$ vertices with $\delta(G)\ge cn$ and $\chi$ is a proper edge-colouring of $G$, then $G$ contains a matching $M$ with at least $n-t$ edges using at least $cn-8t$ distinct colours.
\end{proposition}

The proof proceeds by contradiction. Let $G$ be a balanced bipartite graph on $2n$ vertices with $\delta(G)\ge cn$ where $1/2<c\le 1$, equipped with a proper
edge-colouring $\chi$.
Fix an integer $t$, and let
\begin{equation*}
\mathcal M=\big\{M\subset G:\; M \text{ is a matching with $n-t$ edges}\big\}.
\end{equation*}
By Dirac's theorem, $G$ contains a perfect matching, and hence $\cM\neq\varnothing$ for all $t\leq n-1$. Choose a matching $M\in\cM$ that maximises the number of distinct colours appearing on
its edges. Suppose that $M$ uses fewer than $cn-8t$ distinct colours.

Denote $R:=V(G)\setminus V(M)$, and $|R|=2n-2e(M)=2t.$ 
Let $C(M)$ be the set of colours appearing in $M$ and denote the leftover colours by $C_0:=\{\chi(e):e\in E(G)\}\setminus C(M)$. Fix an arbitrary vertex $v\in R$, we have
\begin{align}\label{eq:N(v)inM}
N_{C_0}(v)\subset V(M),    
\end{align}
where $N_{C_0}(v)=\{u:uv\in E(G){\rm \ and}\ \chi(uv)\in C_0\}$. 
Indeed, otherwise there is an edge $vv'$ with $v,v'\in R$ whose colour lies in $C_0$; then we could add $vv'$ to $M$ and delete an edge of $M$ with a repeated colour (if necessary) to obtain a matching in $\cM$ using more colours, contradicting the maximality of $M$.
Let $C(N(v))$ be the set of colours appearing in the edges incident to $v$. As $N_{C_0}(v)\subset V(M),$ we have
\begin{align}
|N_{C_0}(v)|= |N(v)| - |N_{C(M)}(v)|\ge |N(v)| - |C(M)|\geq 
cn - (cn-8t)= 8t.  \label{eq:|N_{C_0}(v, V(M))|}
\end{align}

The key observation is that, using $C_0$ and \eqref{eq:|N_{C_0}(v, V(M))|}, we can iteratively construct a sequence of colour sets and, more importantly, a collection of rainbow matchings of growing size. After sufficiently many iterations, this forces $M$ to contain a rainbow sub-matching with more than $cn$ edges, contradicting the assumption $|C(M)|\le cn-8t$. To be precise, we present the following key lemma.

\begin{lemma}\label{lem:matchingkey}
Let $M \in \cM$ be the matching with maximal number of colours and let $C_0, R$ be defined as above. Suppose the number of distinct colours in $M$ is at most $cn - 8t$. Then for any integer $\ell$ satisfying $cn/4t < \ell < t^{1/2}/2$, there exist matchings
$H_1\subset H_2\subset...\subset H_\ell \subset M,$
together with a sequence of colour sets $C_1\subset C_2\subset...\subset C_\ell$ defined by
\[
C_i := C_0 \cup C(H_i), \quad \text{where}\ C(H_i)=\{\chi(e) : e \in E(H_i)\}, \qquad 1 \le i \le \ell,
\]
(set $H_0=\varnothing$ and $C_{-1}=\varnothing$) such that the following holds. 
\begin{enumerate}
\item[(i)] For each $0 \leq i\leq \ell$, $H_i$ is rainbow, that is,
$\chi(e)\neq\chi(e')$
for any distinct $e,e'\in H_i$.

\item[(ii)] For each $0 \le i \le \ell$, $e(H_i) = i\cdot 4t.$
 
\item[(iii)]  For each $0 \le i \le \ell$ and  every edge $xy\in H_i$, 
\begin{align}\label{eq:onesideempty}
\text{either}\qquad N_{C_{i-1}}(x,R)=\varnothing\qquad\text{or}\qquad N_{C_{i-1}}(y,R)=\varnothing.     
\end{align}
Moreover, 
\begin{align}\label{eq:degreesum}
|N_{C_{i-1}}(x, R)| + |N_{C_{i-1}}(y,  R)|\geq 8t^{1/2},
\end{align}
where $N_{C}(v, U)$ denote the set of $C$-neighbours of $v$ in $U$.

\item[(iv)] For every $0 \le i \le \ell$ and every vertex $v \in R$, 
\begin{align}\label{eq:|N_{C_i}(v, V(M_i))| ge s}
N_{C_i}(v) \subset V(M)\qquad\text{and}\qquad |N_{C_i}(v)| \ge (i+2)\cdot 4t. 
\end{align}
\end{enumerate}
\end{lemma}

\begin{proof}
The proof is by induction on $i$. For the base case $i=0$, (i)--(iii) vacuously hold, and (iv) follows from~\eqref{eq:N(v)inM} and ~\eqref{eq:|N_{C_0}(v, V(M))|}. Now let $i\ge 1$ and assume that the conclusions of this lemma hold for all indices $j<i$.  We first prove the following trace-back argument which is the core of this proof.

\begin{claim}\label{claim:traceback}
For each $0\leq j<i$, the following holds.
\begin{enumerate}
    \item If there exists an edge $xy$ with $x, y \notin V(H_{j})$ and $\chi(xy) \in C_{j}$, then there exists a matching $H'$ with $V(H') \subseteq R \cup V(H_{j}) \cup \{x,y\}$, $xy\in E(H')$ and $e(H')=e(H_{j})+1$, containing $1$ more colour in $C_0$ than $H_{j}$;

    \item If there exist two vertex-disjoint edges $xy,zw$ with $x,y,z,w \notin V(H_{j})$, $\chi(xy)\neq\chi(zw)$ and $\chi(xy),\chi(zw)\in C_j$, then there exists a matching $H''$ with $V(H'') \subseteq R \cup V(H_{j}) \cup \{x,y,z,w\}$, $xy,zw\in E(H'')$ and $e(H'')=e(H_{j})+2$, containing $2$ more distinct colours in $C_0$ than $H_{j}$; 
    
    \item If there exists two vertex-disjoint edges $xy,zw$ with $x,z\notin V(H_{j})$ and $yw\in E(H_j)$, $\chi(xy)\neq\chi(zw)$ and $\chi(xy),\chi(zw)\in C_j$, then there exists a matching $H''$ with $V(H'') \subseteq R \cup V(H_{j}) \cup \{x,y,z,w\}$, $xy,zw\in E(H'')$ and $e(H'')=e(H_{j})+1$, containing $2$ more distinct colours in $C_0$, and losing at most $1$ colour in $C(H_{j})$.
    
\end{enumerate}
\end{claim}

\begin{poc}
We apply a trace-back method, based on the iterative conditions \eqref{eq:degreesum} and \eqref{eq:onesideempty}. Without loss of generality we just consider $j=i-1$.

For (1), assume $\chi(xy)\in C_{i-1}$. By the definition of $C_{i-1}$, if $\chi(xy)\in C_0$, then simply set $H'=\{xy\}\cup H_{i-1}$. Otherwise, there exists a minimal $m_1$ with $1\leq m_1\leq i-1$ and an edge $x_1y_1\in H_{m_1}\setminus H_{m_1-1}$ such that $\chi(x_1y_1)=\chi(xy)$. We then perform a trace-back process. First, add the edge $xy$. By \eqref{eq:degreesum} and \eqref{eq:onesideempty} for $m_1$, there exists at least $8t^{1/2}$ $C_{m_1-1}$-edges with one end $x_1$ or $y_1$ (say $x_1$) and another end in $R$. We can then pick $z_1\in R$ such that $\chi(x_1z_1)\in C_{m_2}\setminus C_{m_2-1}$ for some $0\leq m_2<m_1$.
Remove $x_1y_1$ and add $x_1z_1$. If $m_2=0$, then we are done; otherwise perform the same process on $x_1z_1$ as we did on $xy$. At each step we remove one matching edge and add a new disjoint edge say $x_sz_s$, maintaining the matching property. The point is that $\chi(x_sz_s)\in C_{m_{s+1}}$ and $m_s$ is decreasing in $s$.  After at most $i<t^{1/2}/2$ steps, we introduce a colour from $C_0$, producing a matching $H'$ as required.

For (2), we run a two-round trace-back process. Given $xy$ and $H_{i-1}$, by (1) we can get a matching $H'$ with the mentioned property and avoid the vertices $z,w$ and keep the edge of the same colour with $zw$ (since the process has at most $i<t^{1/2}/2$ steps and swap one edge each step, meanwhile $|R|=2t> t^{1/2}$ and we have at least $8t^{1/2}$ choices for each step). We then run the same process for $H'$ and $zw$ as in (i). Note we can run at most $i<t^{1/2}/2$ steps to finally introduce a colour from $C_0$ (each step we have enough choices to avoid the chosen edges with ends in $R$ and removed edges from the initial $H_i$), to get a matching $H''$ with required properties. Moreover, in the second round we choose the final $C_0$-edge with a colour different from the one obtained in the first round; this is possible since at each step forbidding one colour rules out at most one candidate neighbour (by properness of the colouring), while we always have at least $8t^{1/2}$ choices. 

For (3), we first remove $yw$ from $H_{i-1}$. Then we run a two-round trace-back process for $xy,wz$ and $H_{i-1}\setminus\{yw\}$ as in (2).
\end{poc}

We now construct the desired $H_i$ via some matching $H_i'\subseteq M$ such that $H_{i-1}\subseteq H_i\subseteq H_i'$ always holds. Let $H'_i\subseteq M$ be a maximal sub-matching such that every $xy\in H'_i$ satisfies \eqref{eq:degreesum} for $i$ (with respect to $C_{i-1}$).
Note $H_{i-1}\subset H'_i$, since each edge in $H_{i-1}$ satisfies \eqref{eq:degreesum} for $i-1$, and $C_{i-2}\subset C_{i-1}$.  

First we prove for every edge $xy\in H'_i$, either 
\begin{align}\label{eq:onesideforH'}
N_{C_{i-1}}(x,R)=\varnothing\qquad \text{or}\qquad N_{C_{i-1}}(y,R)=\varnothing.    
\end{align}
Suppose for contradiction that there exists an edge $xy\in H'_i$ such that both $N_{C_{i-1}}(x,R)$ and $N_{C_{i-1}}(y,R)$ are non-empty. As $H_i'$ satisfies \eqref{eq:degreesum} for $i$, we can find two edges $xw,yz$ such that $w,z\in R$, $\chi(xw)\neq\chi(yz)$, and $\chi(xw),\chi(yz)\in C_{i-1}$. If $xy\notin H_{i-1}$, then as $H_{i-1}\subset H_i'$ we must have $x,y\notin V(H_{i-1})$ and we can apply Claim~\ref{claim:traceback} (2) to $H_{i-1}$ and the two edges $xw$ and $yz$ to get a matching $H''$ with $V(H'') \subseteq R \cup V(H_{i-1}) \cup \{x,y,z,w\}$ and $e(H'')=e(H_{i-1})+2$, containing $2$ more distinct colours in $C_0$ than $H_{i-1}$. Set $M''=H''\cup M\setminus (H_{i-1}\cup \{xy\})$, and remove one edge in $M''$ with repeated colour. Observe then $e(M'')=e(M)$ and $|C(M'')|\geq|C(M)|+2-1= |C(M)|+1$, which conflicts with the colour maximality of $M$. If $xy\in H_{i-1}$, applying Claim~\ref{claim:traceback} (3) to $H_{i-1}$ and $xw$, $yz$ gives a conflict as well.

Now we show $e(H'_{i})\geq i\cdot4t$. To do this, we consider $E_{C_{i-1}}(R,V(M))$,  the $C_{i-1}$-edges between $M$ and $R$. On the one hand, for each $v\in R$ the number of $C_{i-1}$-neighbours of $v$ in $V(M)$ satisfies
$$|N_{C_{i-1}}(v, V(M))|=|N_{C_{i-1}}(v)|\geq (i+1)\cdot 4t,$$
where the first equality holds since $N_{C_{i-1}}(v)\subset V(M)$ by the first term of \eqref{eq:|N_{C_i}(v, V(M_i))| ge s} for $i-1$, and the second inequality holds by the second term of \eqref{eq:|N_{C_i}(v, V(M_i))| ge s} for $i-1$. Hence we have
\begin{align}\label{eq:E_{C_{i-1}}1}
|E_{C_{i-1}}(R,V(M))|\geq  |R|\cdot \min_{v\in R}|N_{C_{i-1}}(v, V(M))| \geq|R|\cdot (i+1)\cdot 4t.   
\end{align}
On the other hand, assume $e(H'_i)<i\cdot 4t$. We partition $M=H'_i\cup (M\setminus H'_i)$ to upper bound $|E_{C_{i-1}}(R,V(M))|$ in two parts. First for the $C_{i-1}$-edges between $R$ and $H'_i$,  
$$|E_{C_{i-1}}(R,V(H'_i))|\leq |R|\cdot e(H'_i)< |R|\cdot i\cdot 4t,$$
where the first inequality holds since either $N_{C_{i-1}}(x,R)=\varnothing$ or $N_{C_{i-1}}(y,R)=\varnothing$ for each $xy\in E(H'_i)$ by 
\eqref{eq:onesideforH'}; and for the $C_{i-1}$-edges between $R$ and $M\setminus H'_i$,
$$|E_{C_{i-1}}(R,V(M\setminus H'_i))|
\le 8t^{1/2}\cdot |M\setminus H'_i|
=8t^{1/2}\,((n-t)-e(H'_i))
\le 8t^{1/2}\,n.$$
since no edge in $M\setminus H'_i$ satisfies \eqref{eq:degreesum} by the maximality of $H'_i$. Hence combining the assumption $|R|= 2t$ and $t\geq n^{2/3}$, we have
\begin{align*}
|E_{C_{i-1}}(R,V(M))|&\leq |E_{C_{i-1}}(R,V(H'_i))|+ |E_{C_{i-1}}(R,V(M\setminus H'_i))|\\
&<|R|\cdot 4t\cdot i+8t^{1/2}n<|R|\cdot 4t\cdot (i+1),    
\end{align*}
which conflicts with \eqref{eq:E_{C_{i-1}}1}. 

Therefore $e(H'_i)\geq 4t\cdot i$, and we can select $H_{i-1}\subset H_i\subset H'_i$ to have exact $4t\cdot i$ edges. Note then (ii) holds for $i$, and meanwhile (iii) also holds for $i$ by \eqref{eq:onesideforH'} and $H_i\subset H'_i$. It remains to check (i) and (iv) for $i$.

For (i), we claim that each colour in $H_i$ appears exactly once. Otherwise suppose for some $xy\in E(H_i\setminus H_{i-1})$, there exists an edge $zw\in E(H_i)$ such that $\chi(xy)=\chi(zw)$. If $zw\in E(H_i\setminus H_{i-1})$, by \eqref{eq:degreesum} for $i$ we can find at least $4t^{1/2}$ edges $zq$ that $q\in R$ and $\chi(zq)\in C_{i-1}$, fix one such $zq$. Applying Claim~\ref{claim:traceback} (1) to $zq$ and $H_{i-1}$ gives a matching $H'$ with $V(H') \subseteq R \cup V(H_{i-1}) \cup \{z,q\}$ and $e(H')=e(H_{i-1})+1$, containing $1$ more colour in $C_0$ than $H_{i-1}$. Then let $M'=H'\cup M\setminus (H_{i-1}\cup\{zw\})$, we have
$e(M')=e(M)$ with one more colour than $M$, a contradiction. If $zw\in E(H_{i-1})$, applying Claim~\ref{claim:traceback} (1) to $xy$ and $H_{i-1}$ gives a conflict similarly as well.

Finally we check (iv) for $i$. Fix an arbitrary vertex $v \in R$. We first show $N_{C_i}(v) \subset V(M).$  Suppose to the contrary that there exists a $C_i$-edge $vz$ such that $z\in R$. If $\chi(vz)\in C_{i-1}$, applying Claim~\ref{claim:traceback} (1) to $vz$ and $H_{i-1}$, we can find a matching $H'$ with $V(H') \subseteq R \cup V(H_{i-1}) \cup \{v,z\}$ and $e(H')=e(H_{i-1})+1$, which contains 1 more colour in $C_0$ than $H_{i-1}$. Set $M'= H'\cup M\setminus H_{i-1}$ (and delete one edge of a repeated colour) to obtain a matching in $\cM$ with one more colour than $M$, a conflict. Otherwise $\chi(vz)\in C_i\setminus C_{i-1}$, which means there exists an edge $xy\in E(H_i\setminus H_{i-1})$ such that $\chi(vz)=\chi(xy)$. Then by (iii) for $i$ we can find one $C_{i-1}$-edge with one end $p\in R\setminus\{v,z\}$ and another end $x$ or $y$ (say $x$). Applying Claim~\ref{claim:traceback} (1) for $xp$ and $H_{i-1}$ gives a conflict similarly. It then remains to check the second inequality in \eqref{eq:|N_{C_i}(v, V(M_i))| ge s} for $i$.
Fix $v\in R$. Note that
$$|N_{C_i}(v)| =|N_{C_0\cup C(H_i)}(v)|
= |N(v)| - |N_{C(M)\setminus C(H_i)}(v)|
\geq |N(v)| - |C(M)\setminus C(H_i)|.$$
Since $H_i$ is a rainbow matching by (i) for $i$, we have $|C(M)\setminus C(H_i)|=|C(M)|-e(H_i)$ and so
$$|N_{C_i}(v)|\geq |N(v)|- \big(|C(M)|-e(H_i)\big)\geq cn - (cn-8t - i\cdot 4t)\geq (i+2)\cdot 4t,$$
where $e(H_i)=i\cdot 4t$ by (ii) for $i$. This completes the proof.
\end{proof}

Based on Lemma~\ref{lem:matchingkey}, we can now prove Proposition~\ref{prop:matching}.

\begin{proof}[Proof of Proposition~\ref{prop:matching}]
Assume $M$ is a matching in $\cM$ with the maximal number of distinct colours, which is at most $cn-8t$. Choose an integer $\ell$ satisfying $cn/4t<\ell<t^{1/2}/2$. By Lemma~\ref{lem:matchingkey}, there exist a rainbow matching $H_\ell\subset M$ such that 
\[
e(H_\ell)=\ell\cdot 4t>cn.
\]
Therefore $M$ contains more than $cn$ distinct colours, contradicting $|C(M)|\le cn-8t$.
\end{proof}

\subsection{Proof of Theorem~\ref{thm:bipartite}}
The following lemma completes the proof of Theorem~\ref{thm:bipartite}.
\begin{lemma}
Let $0<\varepsilon\leq1/2$, $n\in\mathbb{N}$ be sufficiently large and $t=t(n)<\varepsilon n/4$. Let $G$ be a bipartite graph with two parts $X,Y$ that $|X|=|Y|=n$, and $\delta(G)\ge (1/2+\varepsilon)n$, and let $M\subset G$ be a matching with $e(M)\geq n-t$. Then there exists a perfect matching $M'\subset G$ such that 
$$|E(M')\cap E(M)|\geq n-2t.$$
\end{lemma}

\begin{proof}
For every pair $x\in X, y\in Y$ that $x,y\notin V(M)$, we claim there exists at least $2\varepsilon n-2t$ edges $ab\in M$ such that $ay,bx\in E(G)$.
Indeed, by the minimum degree condition, $x$ is adjacent to a set $Y'\subseteq V(M)\cap Y$ of size at least $(1/2+\varepsilon)n-t$. Let $X'\subseteq X$ be the set of vertices matched to $Y'$ by $M$, then $|N(y)\cap X'|\geq 2\varepsilon n-2t.$ Then all edges in $M$ incident to $N(y)\cap X'$ are as desired. 

Now for each unmatched vertices $x,y\not\in V(M)$, we can greedily swap an edge $ab$ as above with edges $ay,bx$ to complete $M$ to a perfect matching $M'$. As we swap at most $n-e(M)\le t$ edges in $M$ during this process, $|E(M')\cap E(M)|\geq n-2t,$
as required.
\end{proof}

Applying the above lemma to the colourful matching we obtained from Proposition~\ref{prop:matching} with $t=O(n^{2/3})$, we obtain a perfect matching $M'$ of $G$ with at least $cn-8t-2t=cn-O(n^{2/3})$.

\medskip

\section{Non-extendable linear forests and matchings}\label{sec:badconstruction}
In this section, we prove Propositions~\ref{Prop:construction} and \ref{Prop:construction2}.

\begin{proof}[Proof of Proposition~\ref{Prop:construction}]
Let $V(G)=A\cup B$ be a partition with $|A|=(1-c)n$ and $|B|=cn$.
Let $G[A]$ be an independent set, let $G[B]$ be a clique, and include all edges
between $A$ and $B$.
Note that
$\delta(G)=|B|=cn$.

Now let $F$ be any spanning linear forest in $G[B]$ (so $V(F)=B$), for example a
collection of vertex-disjoint paths covering $B$.
Consider any Hamilton cycle $H$ in $G$.
Since $A$ is independent, $H$ cannot contain an edge with both endpoints in $A$.
Thus every vertex of $A$ must be incident in $H$ to two edges going from $A$ to
$B$, and consequently $H$ uses exactly $2|A|=2(1-c)n$ edges between $A$ and $B$.
As $H$ has $n$ edges in total, it follows that $H$ contains at most
\[
n-2|A|=n-2(1-c)n=(2c-1)n
\]
edges with both endpoints in $B$.
In particular, since $E(F)\subseteq E(G[B])$, we have $|E(H)\cap E(F)|\le (2c-1)n$,
as required.
\end{proof}

\begin{proof}[Proof of Proposition \ref{Prop:construction2}]
Let $G$ be a balanced bipartite graph with parts $X$ and $Y$, where $|X|=|Y|=n$. Let $X_1 \subseteq X$ and $Y_1 \subseteq Y$ satisfy $|X_1|=|Y_1|=cn$, and set $X_2 := X \setminus X_1$ and $Y_2 := Y \setminus Y_1$. 
Define $G$ by including all edges between $X_1$ and $Y$, and all edges between $X_2$ and $Y_1$, while adding no edges between $X_2$ and $Y_2$. Notice then $\delta(G) \ge cn$.

Let $M$ be a perfect matching between $X_1$ and $Y_1$, so $e(M)=cn$. Now consider any perfect matching $M_1$ in $G$. 
Since there are no edges between $X_2$ and $Y_2$, in order for every vertex in $X_2$ to be covered by an edge of $M_1$, each vertex of $X_2$ must be matched to a distinct vertex in $Y_1$. Consequently, the edges of $M$ incident to these vertices of $Y_1$ cannot belong to $M_1$. Therefore, $M_1$ can contain at most
\[
cn-|X_2| = cn-(1-c)n = (2c-1)n
\]
edges of $M$, as required.
\end{proof}

\medskip

\section{Concluding remarks}\label{sec:conclusion}
In this paper, we prove that in any properly edge-coloured Dirac graphs, one can find an almost maximally colourful Hamilton cycle (perfect matching). We believe that the conclusions in~\Cref{thm:main,thm:bipartite} hold already at $c=1/2$. It would be interesting to settle this case.

Beyond the existence of a single colourful Hamilton cycle, it is natural to ask whether many such cycles exist. In the uncoloured setting, Christofides, K\"uhn and Osthus~\cite{CKO} showed that every $n$-vertex graph $\delta(G)\geq(1/2+\varepsilon)n$ contains at least $n/8$ edge-disjoint Hamilton cycles; see also~\cite{MPS2019,KKKO} for related results. Motivated by these results, we propose the following conjecture.

\begin{conjecture}
Let $1/2<c\le1$. If $G$ is a graph on $n$ vertices with $\delta(G)\ge cn$, then every proper edge-colouring of $G$ contains at least $n/8$ Hamilton cycles, each using at least $cn-o(n)$ distinct colours.
\end{conjecture}

Another natural direction is to study other colourful spanning structures.
In the complete graph setting, Montgomery, Pokrovskiy and Sudakov~\cite{MPS} showed that any tree on $n-o(n)$ vertices can be embedded in a properly edge-coloured $K_n$ with all edges receiving distinct colours.
For Dirac graphs, however, embedding arbitrary almost-spanning trees is impossible without further restrictions.
Koml\'os, S\'ark\"ozy and Szemer\'edi~\cite{KSS} proved that every $n$-vertex tree with maximum degree at most $cn/\log n$ embeds into any graph of minimum degree at least $(1/2+\varepsilon)n$.
This motivates the following colourful analogue.

\begin{conjecture}
Let $1/2<c\le1$. Suppose that $G$ is a graph on $n$ vertices with minimum degree
$\delta(G)\ge cn$. Then there exists a constant $d>0$ such that for every tree $T$ on $n$ vertices with $\Delta(T)\le dn/\log n$, any proper edge-colouring of $G$
contains a copy of $T$ using at least $cn-o(n)$ distinct colours.
\end{conjecture}

\medskip

\newcommand{\etalchar}[1]{$^{#1}$}

\end{document}